\documentclass[12pt,leqno]{amsart}

\usepackage{graphicx}
\usepackage{pinlabel} %for algebraic and geometric topology
\usepackage{amsmath}
\usepackage{amsfonts}
\usepackage{amssymb}
\usepackage{mathtools}
\usepackage{amscd}
\usepackage[all,cmtip]{xy}
\usepackage{amsthm}
\usepackage{comment}
\usepackage{epsfig}
\usepackage[utf8]{inputenc}
\usepackage[colorinlistoftodos]{todonotes}
\usepackage[capitalise]{cleveref}
\usepackage{geometry}
\makeatletter
\providecommand\@dotsep{5}
\def\listtodoname{List of Todos}
\def\listoftodos{\@starttoc{tdo}\listtodoname}
\makeatother

\newtheorem{theorem}{Theorem}[section]

\newtheorem{corollary}[theorem]{Corollary}
\newtheorem{lemma}[theorem]{Lemma}

  \theoremstyle{definition}
\newtheorem{definition}[theorem]{Definition}

\newtheorem{notation}[theorem]{Notation}
\newtheorem{remark}[theorem]{Remark}
\newtheorem{question}[theorem]{Question}

\newcommand{\dbZ}{\mathbb{Z}}

\newcommand{\Z}{\mathbb{Z}}

\newcommand{\calF}{{\mathcal F}}

\newcommand{\calG}{{\mathcal G}}

\newcommand{\calM}{{\mathcal M}}

\newcommand{\fin}{\mathcal{FIN}}

\newcommand{\beq}{\begin{equation}}
\newcommand{\eeq}{\end{equation}}
\DeclareMathOperator{\Irr}{Irr}

\DeclareMathOperator{\rank}{rank}
\DeclareMathOperator{\Ind}{Ind}

\begin{document}

\title[]{On the $K$-theory of groups of the form $\dbZ^n\rtimes \dbZ/m$ with $m$ square-free}


%    General info
\subjclass[2020]{}

\author{Luis Jorge S\'anchez Salda\~na}
\address{Departamento de Matem\' aticas, Facultad de Ciencias, Universidad Nacional Aut\'onoma de M\'exico, M\'exico D.F., Mexico}
\email{luisjorge@ciencias.unam.mx}
%\thanks{The author was supported by DGAPA-UNAM postdoctoral grant.}

\author{Mario Velásquez}
\address{Departamento de Matem\'aticas\\Facultad de Ciencias\\Universidad Nacional de Colombia, sede Bogot\'a\\Cra. 30 cll 45 - Ciudad Universitaria\\ Bogot\'a, Colombia}
\email{mavelasquezme@unal.edu.co}

\date{}


%\keywords{$K$-Theory, Farrell-Jones conjecture, 3-manifold groups, Whitehead groups}

\begin{abstract}
We provide an explicit computation of the topological $K$-theory groups $K_*(C_r^*(\Z^n\rtimes \Z/m))$ of semidirect products of the form $\dbZ^n\rtimes \dbZ/m$ with $m$ square-free. We want to highlight the fact that we are not impossing any conditions on the $\Z/m$-action on $\Z^n$. This generalizes previous computations of Lück-Davis and Langer-Lück.
\end{abstract}
\maketitle

\section{Introduction}

Given a discrete group $\Gamma$, it is a classical problem in noncommutative geometry to compute the topological $K$-theory of the reduced $C^*$-algebra $C^*_r(\Gamma)$. The Baum-Connes conjecture states that $K_*(C^*_r(\Gamma))$ is isomorphic to the $\Gamma$-equivariant $K$-homology group $K_*^\Gamma(\underline{E}\Gamma)$, where $\underline{E}\Gamma$ is the classifying space of $\Gamma$ for proper actions. Once the conjecture is true for $\Gamma$, the isomorphisms allows us to use tools from algebraic topology to compute $K_*(C^*_r(\Gamma))$.

\begin{notation}\label{notation}
    Let $\{p_1,\dots ,p_\ell\}$ be a nonempty set of distinct primes numbers and denote $m=p_1\cdots p_\ell$. Throughout the paper $G$ will denote the group $\Z/p_1 \times \cdots \times \Z/p_\ell\cong \Z/m$. From now on we fix $p\in \{p_1,\dots ,p_\ell\}$ and we denote by $G_p$ the quotient $G/(\Z/p)$.
\end{notation}

In this paper we work semidirect products of the form $\Z^n\rtimes G$, where the $G$-action of $\Z^n$ does not have any restrictions. Our main result is an explicit computation of the topological $K$-theory of groups of the reduced $C^*$-algebra $C^*_r(\Z^n\rtimes G)$. In order to achieve this, we first prove in \Cref{thm:main:torsionfree} that $K_*(C^*_r(\Z^n\rtimes G))$ is a finitely generated and torsion free (which is the main result of Section 3), thus it is enough to compute the rank of this abelian group which is done in the following theorem (which is the main result of Section 4).

\begin{theorem}[\Cref{main:rank:formula}]\label{Intro:main:rank:formula}
 Let $h$ be a fixed generator of $G$, and let $t\geq 1$. Let $e_s$ be the $s$-th elementary symmetric polynomial in $n$-variables. Denote by $\bar y_{G}^t$ (resp. $\bar y_{G_p}^t$) the $n$-tuple of eigenvalues (including multiplicities) of $h^t\otimes \mathrm{Id}:\Z^n\otimes\mathbb C\to \Z^n\otimes\mathbb C$ (resp. $h^{pt}\otimes \mathrm{Id}$).
    Then, for $\alpha\in\{0,1\}$, we have
\begin{align*}
    \rank(K_\alpha(C_r^*(\Z^n\rtimes G))=\left(\frac{1}{m/p}\sum_{s-\alpha\text{ even }}\left(\sum_{t=1}^{m/p}e_s(\bar{y}_{G_p}^t)\right)\right)^{k_l}+\frac{1}{m}\sum_{s-\alpha\text{ even }}\left(\sum_{t=1}^{m}e_s(\bar{y}_G^t)\right)
\end{align*}
where $k_l=\frac{p\left(p^{\frac{l}{p-1}}-1\right)}{m}$ and $l=n-\rank((\Z^n)^{\Z/p})$.
\end{theorem}

Let us put our results into context. In \cite{DL13} Davis and Lück provide a computation of   $K_*(C^*_r(\Z^n\rtimes G))$ provided $G$ is a cyclic group of prime order and the $G$-action on $\Z^n$ is free outside the origin, i.e. $G$ acts freely in $\Z^n-\{0\}$. On the other hand, in \cite{LL12} Langer and Lück generalized the previous computation when $G$ is any cyclic group and still the action of $G$ on $\Z^n$ is free outside the origin. Our \Cref{thm:main:torsionfree} and \cref{Intro:main:rank:formula} generalize widely Lück--Davis and partially Lück--Langer results. The present paper can be considered as a sequel for \cite{SSV24} where a computation of the cohomology groups $H^*(\Z^n\rtimes G)$ was done, it is worth saying in that paper the computations are more intricate since these cohomology groups have torsion. Our computations are related to the $K$-theory the toroidal quotient orbifolds $T^n/G$, see \cite{ADG11}.

In view of our main results and those in \cite{LL12} we pose the following question.

\begin{question}\label{intro:question}
    Is $K_*(C^*_r(\Z^n\rtimes G))$ torsion free whenever $G$ is \emph{any} finite cyclic group and there is no restrictions for the $G$-action on $\Z^n$?
\end{question}

The above question seems to be related to a conjecture posed by Adem-Ge-Pan-Petrosyan in \cite{AGPP} about the collapse of the Lyndon-Hochshild-Serre spectral sequence for the chomology of $\Z^n\rtimes G$, which turned out to be false. The counterexamples for this conjecture might be a good source of counterexamples for \cref{intro:question}.

\subsection*{Acknowledgements.}
We are grateful for the financial support of DGAPA-UNAM grant PAPIIT IA106923.

\section{Preliminaries}

\subsection{Classifying spaces for families and push-out constructions}
 Given a group $\Gamma$, we say a collection $\calF$ of subgroups of $\Gamma$ is a \emph{family} if it is non-empty, closed under conjugation and under taking subgroups. Fix a group $\Gamma$ and a family $\calF$ of subgroups of $\Gamma$. We say that a $\Gamma$-CW-complex $X$ is a \emph{model for the classifying space} $E_{\calF}\Gamma$ if every isotropy  group of $X$ belongs to the family $\calF$ and the fixed point set $X^H$ is contractible whenever  $H$ belongs to $\calF$.

 We will need the following result of L\"uck-Weiermann:
 \begin{theorem} \cite[Corollary 2.8]{LW}\label{LW}
     Let $\calF\subseteq \calG$ be families of subgroups of a group $\Gamma$ such that every element in $\calG-\calF$ is contained in a unique maximal element in $\calG-\calF$. Let $\calM$ be a complete system of representatives of the conjugacy classes
of subgroups in $\calG-\calF$ which are maximal in $\calG-\calF$. Let $\mathcal{SUB}(M)$ be the family of subgroups of $M$. Then, there is a cellular $\Gamma$-pushout
\begin{equation}\label{pushout}
\xymatrix{\bigsqcup_{M\in\calM} \Gamma\times_{N_{\Gamma}M}E_{\calF\cap N_\Gamma M} (N_\Gamma M)\ar[r]^-i\ar[d]^{\lambda}&E_{\calF}(\Gamma)\ar[d]\\
\bigsqcup_{M\in\calM} \Gamma\times_{N_{\Gamma}M}E_{\mathcal{SUB}(M)\cup(\calF\cap N_\Gamma M)} (N_\Gamma M)\ar[r]&X
}
\end{equation}
such that $X$ is a model for $E_{\calG}(\Gamma)$.
 \end{theorem}

 The following lemma will be also useful.
\begin{lemma}\cite[Lemma~4.4]{DQR11}\label{lemma:union:families}
Let $\Gamma$ be a group and $\calF,\, \calG$ be two families of subgroups of $\Gamma$.  Choose arbitrary $\Gamma$-$CW$-models for $E_{\calF} \Gamma$, $E_{\calG}\Gamma$ and $E_{\calF\cap\calG} G$. Then, the $\Gamma$-$CW$-complex $X$ given by the cellular homotopy $\Gamma$-pushout
\[
\xymatrix{
E_{\calF\cap\calG}\Gamma \ar[r] \ar[d] & E_{\calF}\Gamma \ar[d]\\
E_{\calG}\Gamma \ar[r] & X
}
\]
is a model for $E_{\calF\cup\calG}\Gamma$.
\end{lemma}

\subsection{Splitting theorem of Angel-Becerra-Velásquez}
For the proof of our main result we will need a $K$-theoretic decomposition obtained in \cite{ABV}. Consider an extension of discrete groups 
\begin{equation}\label{ext-decomp}0\to A\to W \to Q\to0,\end{equation}where $A$ is finite. Let $X$ is a proper $Q$-CW-complex, then $X$ is naturally a proper $W$-CW-complex where $A$ acts trivially.

Let $\Irr(A)$ be the set of isomorphism classes of complex irreducible representations of $A$. The group $Q$ acts on $\Irr(A)$ by conjugation. Let $[\rho]\in\Irr(A)$, we denote by  by $Q_{[\rho]}$ to the isotropy group of $[\rho]$ under the $Q$-action.

%\begin{theorem}[Remark 2.3 and Prop. 2.4 in \cite{ABV}]
% There is a $U(1)$-central extension    
%\begin{equation}\label{s1-central}
%    0\to U(1)\to\widetilde{Q}_{[\rho]}\to Q_{[\rho]}\to 0.
%\end{equation}  Such that $\rho$ can be extended to a representation of $\Gamma_{[\rho]}$ if and only if the above extension splits. Moreover if $A$ is abelian the sequence always splits.
%\end{theorem}
%If $\widetilde{Q}_{[\rho]}$ is a group as in the extension (\ref{s1-central}) and $X$ is a proper $Q_{[\rho]}$-CW-complex, then we can define the $Q_{[\rho]}$-equivariant $\widetilde{Q}_{[\rho]}$-twisted K-theory groups of $X$ as the Grothendieck group of the monoid of isomorphism classes of $\widetilde{Q}_{[\rho]}$-vector bundles over $X$ where $U(1)$ acts by multiplication by inverses. Define for $n\geq0$ $${}^{\widetilde{Q}_{[\rho]}}{K}_{Q_{[\rho]}}^{-n}(X)=\ker\left({}^{\widetilde{Q}_{[\rho]}}{K}_{Q_{[\rho]}}(X\times S^n)\xrightarrow{i^*}{}^{\widetilde{Q}_{[\rho]}}{K}_{Q_{[\rho]}}(X)\right),$$ where $Q_{[\rho]}$ acts trivially on $S^n$ and $i(x)=(x,e_1)$ with $e_1$ the first element in the canonical basis of $\mathbb{R}^{n+1}$. The groups for $n< 0$ are defined using Bott periodicity.

%Note that when the extension (\ref{s1-central}) splits we have a canonical isomorphism
%$${}^{\widetilde{Q}_{[\rho]}}{K}_{Q_{[\rho]}}^{-n}(X)\cong K_{Q_{[\rho]}}^{-n}(X).$$
%We have the following result.

    As a direct consequence of Theorem 3.3 in \cite{ABV} we get the following result.

\begin{theorem}\label{decompintro}
		Let $W$ be a discrete group  as in extension (\ref{ext-decomp}) , and $X$ be a proper $Q$-CW-complex. Suppose that $A$ is abelian. There is a natural isomorphism 
		\begin{equation*}\Psi_X:K^*_W(X)\to\bigoplus_{[\rho]\in \Irr(A)}K^*_{Q_{[\rho]}}(X).\end{equation*} This isomorphism is functorial on $G$-maps $X\to Y$ of proper $W$-CW-complexes on which $A$ acts trivially. 
	\end{theorem}

Now we will extend the above result for equivariant $K$-homology for actions of finite groups. For this we need the following consequence of the equivariant version of the Spanier-Whitehead duality. For details see  Theorem III.4.1 and Remark III.4.2 in \cite{MLS}.

\begin{theorem}\label{SW-dual}Let $G$ be a finite group. Let $X$ be a finite $G$-CW-complex, suppose $X$ is equivariantly embedded as a neighborhood retract 
in a complex representation $V$. Then we have a natural isomorphism
\[
  K_*^G(X)\cong\widetilde{K}_G^{*-1}(S^V-X).  
\]
\end{theorem}
 We have the following decomposition for equivariant $K$-homology.

\begin{theorem}\label{decompintro-hom}
		Let $G$ be a finite group, let $A\subseteq G$ be a normal abelian subgroup, let us denote by $Q$ the quotient group. Let $X$ be a proper $Q$-CW-complex. Then there is a natural isomorphism 
		\begin{equation*}\Psi_X:K_*^G(X)\to\bigoplus_{[\rho]\in \Irr(A)}K_*^{Q_{[\rho]}}(X).\end{equation*} This isomorphism is functorial on $G$-maps $X\to Y$ of  $G$-CW-complexes on which $H$ acts trivially. 
	\end{theorem}
\begin{proof}
    Note that if $X$ is a finite $G$-CW-complex such that the $A$-action on $X$ is  trivial, then it can be embedded in a complex representation $V$ of $G$ where the subgroup $A$ acts trivially. Then we can apply Theorem \ref{decompintro} to $S^V-X$ and by Theorem \ref{SW-dual} we get the desired result.
\end{proof}

\section{Torsion freeness of $K_*(C_r^*(\Z^n\rtimes G))$}

In all the results throughout this section we follow the notation established in \Cref{notation}. The main goal of this section is to prove that the $K$-theory group $K_*(C_r^*(\Z^n\rtimes G))$ is torsion free. In order to achieve this, first we prove that, given a prime $p$ that divides $m$ the $K$-theory group $K_*^{\Z^n\rtimes \Z/p}(\underline{E}(\Z^n\rtimes \Z/p))$ is torsion free. Next, via \Cref{coro:induction}, we use the previously mentioned result to prove that $K_*^{\Z^n\rtimes G}(\underline{E}(\Z^n\rtimes G))$ is torsion free. The main result follows as a consequence of the Baum-Connes conjecture.

\subsection{Torsion freeness of $K_*^{\Z^n\rtimes \Z/p}(\underline{E}(\Z^n\rtimes \Z/p))$}
First specialize to the case $\Z^n\rtimes \Z/p$ considered as a subgroup of $\Z^n\rtimes G$. Considering $\Z^n$ as a $\Z/p$-module, we have the following extension of $\Z/p$-modules
$$0\to (\Z^n)^{\Z/p}\to \Z^n\xrightarrow{\pi} \Z^l \to0$$
where $\Z/p$-acts on $\Z^l$ freely outside the origin. Note that this extension does not split in general. Nevertheless,  we are able to prove a Künneth-type theorem in \Cref{kunneth}.

\begin{theorem}\label{kunneth}
    There is an isomorphism of $G_p$-modules$$K^*_{\Z^n\rtimes \Z/p}(\underline{E}(\Z^n\rtimes \Z/p))\cong K^*(B((\Z^n)^{\Z/p}))\otimes K^*_{\Z^l\rtimes\Z/p}(\underline{E}(\Z^l\rtimes\Z/p)).$$
\end{theorem}
\begin{proof}
    First note that as $\Z^n$ and $\Z^l$ are torsion free we have isomorphisms
$$K^*_{\Z^n\rtimes \Z/p}(\underline{E}(\Z^n\rtimes \Z/p))\cong K_{\Z/p}^*(B\Z^n)$$
$$K^*_{\Z^l\rtimes \Z/p}(\underline{E}(\Z^l\rtimes \Z/p))\cong K_{\Z/p}^*(B\Z^l).$$
Now we use the Segal's spectral sequence for $\Z/p$-equivariant $K$-theory associated to a $\Z/p$-cover, see for instance Proposition 5.2 in \cite{Se68}. Given a  $\Z/p$-cover $\mathcal{U}=\{U_i\}_{i\in\Sigma}$ for $BL$ and  a subset $\sigma\subseteq\Sigma$, define $U_\sigma=\cap_{i\in\sigma}U_i$.
From now on we assume that  $\mathcal{U}=\{U_i\}_{i\in\Sigma}$ is a $\Z/p$-cover such that for all $\sigma$, the set $U_\sigma$ is $\Z/p$-homotopy equivalent to an orbit
$(\Z/p)/H_\sigma$ for a subgroup $H_\sigma\subseteq \Z/p$. The existence of $\mathcal{U}$ is  guaranteed by the work in \cite{AE}. Abusing of notation, we denote $\pi\colon B\Z^n\to B\Z^l$ the map induced by $\pi\colon \Z^n\to \Z^l$. We have that $\mathcal{V}=\pi^{-1}(\mathcal{U})$ is a $\Z/p$-covering of $B\Z^n$, and associated to $\mathcal{V}$ we have an spectral sequence converging to $K_{\Z/p}^*(B\Z^n)$ whose $E_1$-term is given by a Bredon cochain complex with local coefficients (Definition 2.7 in \cite{BEUV}) $$E_1^{\alpha,\beta}=C^{\alpha+\beta}(B\Z^n,\mathcal{V};\mathcal{R})=\bigoplus_{\sigma\subseteq\Sigma, |\sigma|=\alpha+1}K^\beta_{\Z/p}(\pi^{-1}(U_\sigma)).$$With a differential defined using the restriction homomorphisms on equivariant $K$-theory. Note that for every $\sigma\subseteq\Sigma$ we have a $\Z/p$-homotopy equivalence $$\pi^{-1}(U_\sigma)\simeq_{\Z/p}B((\Z^n)^{\Z/p})\times\left( (\Z/p)/H_\sigma\right).$$

Since the $\Z/p$-action on $B(\Z^n)^{\Z/p}$ is trivial, we have an isomorphism of cochain complexes
$$C^{\alpha+\beta}(B\Z^n,\mathcal{V};\mathcal{R})\cong K^\beta(B(\Z^n)^{\Z/p})\otimes \bigoplus_{\sigma\subseteq\Sigma, |\sigma|=\alpha+1}K^0_{\Z/p}(U_\sigma).$$
As a consequence the $E^2$-term is isomorphic to
$$E_2^{\alpha,\beta}\cong K^\beta(B(\Z^n)^{\Z/p})\otimes H^\alpha(B\Z^l;\mathcal{R}),$$where $H^\alpha(B\Z^l;\mathcal{R})$ denotes the $\alpha$-th Bredon cohomology group with coefficients in representations. Moreover by construction we have an isomorphism$$E_\infty^{\alpha+\beta}\cong K^\beta(B(\Z^n)^{\Z/p})\otimes F_\infty^{\alpha,0},$$where $F_\infty^{\alpha,0}$ denotes the $\infty$-term in the usual Atiyah-Hirzebruch spectral sequence for $\Z/p$-equivariant $K$-theory of $B\Z^l$. But as the $\Z/p$-action on $\Z^l$ is free outside the origin, by Theorem 7.1 in \cite{DL13}, we know that $F_\infty^{\alpha,0}$ is torsion free then $E_\infty^{\alpha+\beta}$ is torsion free and there are no extension problems. Therefore the spectral sequence collapses and we have the desired isomorphism.
\end{proof}
\begin{corollary}\label{cor:torsionfree}
    $K_*^{\Z^n\rtimes \Z/p}(\underline{E}(\Z^n\rtimes \Z/p))$ is finitely generated and torsion free.
\end{corollary}
\begin{proof}
Since $K^*(B((\Z^n)^{\Z/p}))$ is the $K$-theory of a torus, it is torsion free. On the other hand $K^*_{\Z^l\rtimes\Z/p}(\underline{E}(\Z^l\rtimes\Z/p))$ is torsion free by \cite[Theorem 7.1]{DL13}. Thus by \Cref{kunneth} we have that $K^*_{\Z^n\rtimes \Z/p}(\underline{E}(\Z^n\rtimes \Z/p))$ is torsion free. Now the result is a consequence of the Universal Coefficient Theorem for equivariant $K$-theory in \cite{JoLu}.
\end{proof}

\subsection{Torsion freeness of $K_*^{\Z^n\rtimes G}(\underline{E}(\Z^n\rtimes G))$}
Denote by $\Z_{(p)}$ the subring of $\mathbb Q$ of all fractions with denominator prime relative with $p$.
We denote the homology theory
$K_*(-)\otimes \Z_{(p)}$ by $K_*(-)_{(p)}$. Let $\Gamma$  be a group given by an extension
\begin{equation}\label{eq:extensions:LGammaH}
    0\to L\to \Gamma\xrightarrow{\pi} H\to 0,
\end{equation}
where $L$ is finitely generated free abelian and $H$ is a subgroup of $G$. Provided $p$ divides $|H|$, let us denote $\Gamma_p=\pi^{-1}(\Z/p)$.

\begin{lemma}\label{lemma:reducing:to:invariants}
 For all $\beta\in \Z$, we have an isomorphism
\[
K_\beta(B\Gamma)_{(p)}\cong \left(K_\beta(B\Gamma_p)_{(p)}\right)_{G_p} \]
where the coinvariants on the right hand side are taken with respect to the $G_p$-action on $K_\beta(B\Gamma_p)_{(p)}$ induced by conjugation of $G_p$ on $\Gamma_p$.
\end{lemma}
\begin{proof}
    The extension induces, after choosing suitable models, a fibration of classifying spaces
\begin{equation}\label{fibration}
    B\Gamma_p\to B\Gamma\to BG_p
\end{equation}
Consider the Leray-Serre spectral sequence associated to this fibration for  $K_*(-)_{(p)}$. By \cite{milnor} this spectral sequence converges to $K_*(B\Gamma)_{(p)}$ with second page:
\begin{equation*}
E^2_{\alpha,\beta}=H_\alpha(G_p;K_\beta(B\Gamma_p)_{(p)})=
    \begin{cases}
\left(K_\beta(B\Gamma_p)_{(p)}\right)_{G_p}&\alpha=0\\
    0&\alpha\neq0.
        \end{cases}
    \end{equation*}
Now the result follows.

\end{proof}

\begin{remark}
It is worth saying that we are working with $K$-homology instead of $K$-theory theories because the Leray-Serre spectral sequence of the fibration (\ref{fibration}) converges to the $K$-homology of the total space, whilst for $K$-theory a $\lim^1$-term shows up.

\end{remark}
Note that \cref{lemma:reducing:to:invariants} gives us a way to get a computation of $K_*(B\Gamma)$ from a computation of $K_*(B\Gamma_p)$ for every $p$ dividing $m$, compare with \cite[Lemma 5.2]{SSV24}.

Before proving the main result of this subsection, we need to state some notation and an auxiliary theorem. From now on, for convenience, assume that $m=p_1\cdots p_\ell$ with $p_1<\cdots< p_\ell$. Recall $G$ is denoting a finite cyclic group of order $m$.  Let us filter the family $\fin$ as follows
\begin{equation}\label{filtration}
   \mathcal{F}_0=\{1\}\subseteq \mathcal{F}_1\subseteq\cdots\subseteq\mathcal{F}_\ell=\fin 
\end{equation}
 where $\calF_i$ is the family of all subgroups of $\Gamma$ whose order divide $p_1\cdots p_i$.

Note that every element, if any, in $\calF_{i+1}-\calF_i$ is maximal, hence we have everything set up to apply \cref{LW} to $\calF_{i}\subseteq\calF_{i+1}$.
\begin{definition}
   Let $\Gamma$ be as above, let $p$ be a prime number that divides $|H|$, and let $\calF$ be a family of subgroups of $\Gamma$. We say that $\Gamma$ satifies the ($p$, $\calF$)-condition if  the induction map $\Ind_{\Gamma_p}^\Gamma( E_{\calF\cap\Gamma_p}\Gamma_p)\to E_\calF \Gamma$ leads to an isomorphism
\begin{equation}\label{coinvariant}
    \left(K_*^{\Gamma_p}(E_{\calF\cap\Gamma_p}\Gamma_p)_{(p)}\right)_{H_p}\xrightarrow{\Ind_{\Gamma_p}^{\Gamma}} K_*^{\Gamma}(E_\calF\Gamma)_{(p)},
    \end{equation}where $\Gamma_p=\pi^{-1}(\Z/p)$ and $H_p=H/(\Z/p)$. 
\end{definition}

Note that the map (\ref{coinvariant}) is well defined, that is the map induced by $\Ind_{\Gamma_p}^\Gamma( E_{\calF\cap\Gamma_p}\Gamma_p)\to E_\calF \Gamma$ factors through coinvariants, since any generator of $H_p$ acts trivially on the group $K_*^{\Gamma}(\Ind_{\Gamma_p}^\Gamma( E_{\calF\cap\Gamma_p}\Gamma_p))\cong K_*^{\Gamma_p}(E_{\calF\cap\Gamma_p}\Gamma_p)$, for details see Section 1 in \cite{Luck2002}.

\begin{theorem}\label{induction} Let $p$ be a prime that divides m. Then $\Gamma$ satisfies the ($p$,$\calF_i$)-condition for every   $i\in\{0,\ldots,\ell\}$.
\end{theorem}
\begin{proof}
     The proof is by induction on $i$. For $i=0$, it is just the isomorphism given by \cref{lemma:reducing:to:invariants}. The induction hypothesis consists on assuming that all possible extensions of the form \eqref{eq:extensions:LGammaH} satisfy the $(p,\calF_i)$-condition.
    For the inductive step,  apply \Cref{LW} for $\Gamma$ and the families $\calF_i\subseteq \calF_{i+1}$ and consider also the $\Gamma_p$-pushout given by restricing the action from $\Gamma$ to $\Gamma_p$. The induction functor $\Ind_{\Gamma_p}^\Gamma$ gives us a morphism on the associated Mayer-Vietoris
    exact sequences. The idea is to use the five lemma to obtain the desired isomorphism. Let $M$ be a maximal element in $\calF_i\subseteq \calF_{i+1}$, then we have that $N_\Gamma M$ fits in an extension
    \begin{equation}\label{normalizer-ext}0\to L^M\to N_\Gamma M\to \bar M\to0,\end{equation}
    where $\bar M$ is a subgroup of $G$.
    Then by induction hypothesis, condition $(p,\calF_i\cap N_\Gamma M)$ provides the following isomorphism
$$\left(K_*^{\Gamma_p}(E_{\calF_i\cap N_\Gamma M\cap\Gamma_p }((N_\Gamma M)\cap\Gamma_p)_{(p)}\right)_{\bar M_p}\xrightarrow{\Ind_{\Gamma_p}^{\Gamma}} K_*^{\Gamma}(E_{\calF_i \cap N_\Gamma M}(N_\Gamma M))_{(p)}$$
 On the other hand we have a map
\begin{equation*}
 \left(K_*^{\Gamma_p}(E_{\mathcal{SUB}(M)\cup(\calF_i\cap N_\Gamma M\cap\Gamma_p )}((N_\Gamma M)\cap\Gamma_p)_{(p)}\right)_{\bar M_p}\xrightarrow{\Ind_{\Gamma_p}^{\Gamma}} K_*^{\Gamma}(E_{\mathcal{SUB}(M)\cup(\calF_i \cap N_\Gamma M)}(N_\Gamma M))_{(p)}   
\end{equation*}
we want to show that this map is an isomorphism. By \cref{lemma:union:families}, we have a $\Gamma$-pushout
\begin{equation}\label{union}
\xymatrix{E_{\mathcal{SUB}(M)\cap(\calF_i\cap N_\Gamma M)} (N_\Gamma M)\ar[r]^-i\ar[d]^{\lambda}&E_{\calF_i\cap N_\Gamma M}(N_\Gamma M)\ar[d]\\
E_{\mathcal{SUB}(M)} (N_\Gamma M)\ar[r]&E_{\mathcal{SUB}(M)\cup(\calF_i \cap N_\Gamma M)
}(N_\Gamma M)}\end{equation}
In both cases $\mathcal{SUB}(M)$ and $\mathcal{SUB}(M)\cap\calF_i\cap N_\Gamma M$ have a maximum element (here we are using the ordering  of the $p_j$'s), $M$ and $M_{i+1}=M/(\Z/p_{i+1})$ respectively, then $E(W_\Gamma M)$ with the action induced by the quotient map is a model for $E_{\mathcal{SUB}(M)}N_\Gamma M$ and $E(N_\Gamma M/(M_{i+1}))$ with the action induced by the quotient map is a model for $E_{\mathcal{SUB}(M)\cap\calF_i\cap N_\Gamma M}N_\Gamma M$. 

%\textcolor{red}{Aquí hay una pequeña imprecisión porque $M\subseteq N_\Gamma M$ no es central, luego la descomposición no es tan sencilla porque los grupos de isotropía pueden variar, sin embargo se puede usar la hipótesis de inducción para cada grupo de isotropía, además hay que usar la descomposición en K-homología. Para esto hay que reducir el cálculo de K-homología a grupos finitos partiendo el normalizador por $L^M$}

Now as $L^M$ embeds into $W_\Gamma M$, then $L^M$ acts freely on $EW_\Gamma M$ and we have an isomorphism
$$K_*^{N_\Gamma M}(EW_\Gamma M)\cong K_*^{\bar{M}}(EW_\Gamma M/L^M).$$
On the other hand, $M$ maps injectively into $\bar M$ which acts trivially on $EW_\Gamma M/L^M$, then by Theorem \ref{decompintro-hom} we have an isomorphism
\begin{align*}  
K^{N_\Gamma M}_*(EW_\Gamma M)&\cong K_*^{\bar{M}}(EW_\Gamma M/L^M)\\&\cong\bigoplus_{[\rho]\in \Irr(M)}K_*^{(\bar M/M)_{[\rho]}}(EW_\Gamma M/L^M).\end{align*}
Note that, via the quotient map $W_\Gamma M\to \bar M/M$, we can define a conjugacy action of $W_\Gamma M$ on $\Irr(M)$ and given $[\rho]\in\Irr(M)$, $L^M$ is a normal subgroup of $(W_\Gamma M)_{[\rho]}$. Then we have
$$\bigoplus_{[\rho]\in \Irr(M)}K_*^{(\bar M/M)_{[\rho]}}(EW_\Gamma M/L^M)\cong \bigoplus_{[\rho]\in \Irr(M)}K_*^{(W_\Gamma M)_{[\rho]}}(EW_\Gamma M)\cong \bigoplus_{[\rho]\in \Irr(M)}K_*(B(W_\Gamma M)_{[\rho]}).$$

Then combining the above isomophisms we have

$$K_*^{N_\Gamma M}(EW_\Gamma M)\cong \bigoplus_{[\rho]\in \Irr(M)}K_*(B(W_\Gamma M)_{[\rho]})$$

%Where the last isomorphism it is a consequence of $EW_\Gamma M$ is a model for $EW_\Gamma M_{[\rho]}$ and from  the extension
%$$0\to L^M\to (W_\Gamma M)_{[\rho]}\to (G/M)_{[\rho]}\to0.$$
In a similar way we have
\begin{align*}
K^{N_\Gamma M}_*(E(N_\Gamma M/(M_{i+1})))\cong \bigoplus_{[\rho]\in \Irr(M_{i+1})} K_*(B(N_\Gamma M/M_{i+1})_{[\rho]}).
\end{align*}
%Now by the Universal Coefficient Theorem for equivariant K-theory (Thm. 0.3 in \cite{JoLu}) we obtain natural isomorphisms compatible with induction \textcolor{red}{checar esta afirmación}
%\begin{align*}  
%K^{N_\Gamma M}_*(E(W_\Gamma M))&\cong \left(K^{W_\Gamma M}_*(E(W_\Gamma M))\right)^{|M|}
%\\K^{N_\Gamma M}_*(E(N_\Gamma M/(M/\Z/p_{i+1})))&\cong \left(K^{N_\Gamma M/(M/\Z/p_{i+1})}_*(E(N_\Gamma M/(M/\Z/p_{i+1})))\right)^{|M|/p_{i+1}}.
%\end{align*}
Now \cref{lemma:reducing:to:invariants} applied to the $p$-localization of the right hand side of the above isomorphisms proves that $N_\Gamma M$ satisfies conditions
 $(p,\mathcal{SUB}(M))$ and $(p,\mathcal{SUB}(M)\cap\calF_i\cap N_\Gamma M)$. By the five lemma applied to the morphism of Mayer-Vietoris sequences given by restricting the $\Gamma$-pushout (\ref{union}) to $\Gamma_p$ we get that $N_\Gamma M$ satisfies the $(p,\mathcal{SUB}(M)\cup(\calF_i\cap N_\Gamma M))$-condition.

Again, the five lemma applied to the morphism of Mayer-Vietoris sequences given by restricting the $\Gamma$-pushout (\ref{pushout}) to $\Gamma_p$ gives us the desired result. 
\end{proof}

\begin{corollary}\label{coro:induction}
If $\left(K_*^{\Gamma_p}(\underline{E}\Gamma_p)_{(p)}\right)_{G_p}$ is torsion free for every prime $p$ that divides $m$, then 
    $K_*^\Gamma(\underline{E}\Gamma)$ is torsion free. 
\end{corollary}

Now we can state the main result of this section.
\begin{theorem}\label{thm:main:torsionfree} The group $K_*(C_r^*(\Z^n\rtimes G))$ is finitely generated and torsion free.
\end{theorem}
\begin{proof} Let $p$ be a prime that divides $m$.
     As the Baum-Connes conjecture is true for $\Z^n\rtimes G$, by  Theorem  \ref{induction} we have an isomorphism
    $$\left(K_*(C_r^*(\Z^n\rtimes \Z/p))_{(p)}\right)_{G_p}\longrightarrow K_*(C_r^*(\Z^n\rtimes G))_{(p)}.$$
   By Remark A.2 in \cite{DL13}
   $$\left(K_*(C_r^*(\Z^n\rtimes \Z/p))_{(p)}\right)_{G_p}\cong \left(K_*(C_r^*(\Z^n\rtimes \Z/p))_{(p)}\right)^{G_p}.$$But by the above \Cref{cor:torsionfree} and the Baum-Connes conjecture, $K_*(C_r^*(\Z^n\rtimes \Z/p))_{(p)}$ is torsion free and taking invariants we get a torsion free group also. As a conclusion $K_*(C_r^*(\Z^n\rtimes G))$ cannot contain $p$-torsion for every $p\mid m$. Now the result follows.
\end{proof}

\section{Formula for the rank}

In light of \Cref{thm:main:torsionfree}, in order to obtain a full computation of $K_*(C_r^*(\Z^n\rtimes G))$ we only have to compute its rank as finitely generated abelian group. The main goal of this section is to provide such a computation by means of a very explicit formula.

\begin{theorem}\label{NM} Let $p$ be a prime dividing $m$.   Let $\Z^l$ be a $\Z/m$-module, such that the action restricted to $\Z/p$ is free outside the origin. 
\begin{enumerate}
    \item There is an isomorphism $K_1(B\Z^l)^{\Z/p}\to K_1(C_r^*(\Z^l\rtimes\Z/p))$ of $\Z/m$-modules. 

    \item There is an exact sequence of $\Z/m$-modules
$$0\to \bigoplus_{(P)\in\mathcal{P}}\widetilde{R}(P)\to K_0(C_r^*(\Z^l\rtimes\Z/p))\to K_0(\underline{B}(\Z^l\rtimes\Z/p))\to0$$
where $\widetilde{R}(P)$ is the kernel of the map $R(P)\to R(1)$ which sends $[V]$ to $[\mathbb{C}^{\dim(V)}]$.
\end{enumerate}
\end{theorem}
\begin{proof}
    Both statements are proved in  \cite[Theorem 8.1]{DL13} when $\Z^l$ is endowed only with the structure of a $\Z/p$-module. We only need to argue that the homomorphisms are also $\Z/m$-equivariant. 

    For (1): $K_1(B\Z^l)^{\Z/p}\to K_1(C_r^*(\Z^l\rtimes\Z/p))$ the map is induced by the (natural) map $E\Z^l \to \underline E (\Z^l\rtimes\Z/p)$, which is $\Z/m$-equivariant, hence the homomorphism is also $\Z/m$-equivariant.

    For (2): the map $K_0(C_r^*(\Z^l\rtimes\Z/p))\to K_0(\underline{B}(\Z^l\rtimes\Z/p))$ is induced, via the Baum-Connes isomorphism, by the projection $\underline E (\Z^l\rtimes\Z/p) \to \underline B (\Z^l\rtimes\Z/p)$ which is $\Z/m$-equivariant, hence the aforementioned homomorphism is also $\Z/m$-equivariant.
\end{proof}
Next, we compute the rank of $K_*(C_r^*(\Z^n\rtimes G))$. First we need the following lemma.

\begin{lemma}\label{rank} Let $e_s$ be the $s$-th elementary symmetric polynomial in $n$-variables. Let $H$ be a finite cyclic group acting on $\Z^n$,  let $h$ be a fixed generator of $H$, and let $t\geq 1$ denote by $\bar y_H^t$ the $n$-tuple of eigenvalues (including multiplicities) of $h^t\otimes \mathrm{Id}:\Z^n\otimes\mathbb C\to \Z^n\otimes\mathbb C$.
    For each $s\geq 0$, we have
    \[\rank(H_s(\Z^n)_{H}))= \frac{1}{|H|}\sum_{t=1}^{|H|}e_s(\bar{y}_H^t). \]
\end{lemma}
\begin{proof}
    This result was proved in Theorem 8.2 in \cite{SSV24} for cohomology. The statement follows as a standard application of the Universal Coefficient Theorem, and the duality of invariants and coinvariants functors (see Remark A.2 in \cite{DL13}).
\end{proof}

\begin{theorem}\label{main:rank:formula}
 Let $h$ be a fixed generator of $G$, and let $t\geq 1$. Let $e_s$ be the $s$-th elementary symmetric polynomial in $n$-variables. Denote by $\bar y_{G}^t$ (resp. $\bar y_{G_p}^t$) the $n$-tuple of eigenvalues (including multiplicities) of $h^t\otimes \mathrm{Id}:\Z^n\otimes\mathbb C\to \Z^n\otimes\mathbb C$ (resp. $h^{pt}\otimes \mathrm{Id}$).
    Then, for $\alpha\in\{0,1\}$, we have
\begin{align*}
    \rank(K_\alpha(C_r^*(\Z^n\rtimes G))=\left(\frac{1}{m/p}\sum_{s-\alpha\text{ even }}\left(\sum_{t=1}^{m/p}e_s(\bar{y}_{G_p}^t)\right)\right)^{k_l}+\frac{1}{m}\sum_{s-\alpha\text{ even }}\left(\sum_{t=1}^{m}e_s(\bar{y}_G^t)\right)
\end{align*}
where $k_l=\frac{p\left(p^{\frac{l}{p-1}}-1\right)}{m}$ and $l=n-\rank((\Z^n)^{\Z/p})$.
\end{theorem}
\begin{proof}
    
Theorem \ref{induction} implies that $$    \rank(K_\alpha(C_r^*(\Z^n\rtimes G)))=\rank(K_\alpha(C_r^*(\Z^n\rtimes\Z/p))_{G_p}).$$
    On the other hand, as $K_*(C_r^*(\Z^n\rtimes\Z/p))$ is torsion free, Theorem \ref{kunneth}  gives us that $K_\alpha(C_r^*(\Z^n\rtimes\Z/p))_{G_p}$ is isomorphic to 
    \begin{align*}
         &\left(K_\alpha(B(\Z^n)^{\Z/p})\otimes K_0(C_r^*(\Z^l\rtimes\Z/p))\oplus K_{\alpha+1}(B(\Z^n)^{\Z/p})\otimes K_1(C_r^*(\Z^l\rtimes\Z/p))\right)_{G_p}\cong\\
         &\left(K_\alpha(B(\Z^n)^{\Z/p})\otimes K_0(C_r^*(\Z^l\rtimes\Z/p))\right)_{G_p}\oplus \left(K_{\alpha+1}(B(\Z^n)^{\Z/p})\otimes K_1(C_r^*(\Z^l\rtimes\Z/p))\right)_{G_p} 
    \end{align*}
The rank of the first term, $\left(K_\alpha(B(\Z^n)^{\Z/p})\otimes K_0(C_r^*(\Z^l\rtimes\Z/p))\right)_{G_p}$, by \Cref{NM}(2) equals

\begin{align*}
\rank&\Biggl(K_\alpha(B(\Z^n)^{\Z/p})\otimes   \biggl(\bigoplus_{(P)\in\mathcal{P}}\widetilde{R}(P)\oplus K_0(B\Z^l)_{\Z/p}\biggr)\Biggl)_{G_p}\\&
=\rank\Biggl(K_\alpha(B(\Z^n)^{\Z/p})\otimes\bigoplus_{(P)\in\mathcal{P}}\widetilde{R}(P)\Biggr)_{G_p}+\rank \left(K_\alpha(B(\Z^n)^{\Z/p})\otimes K_0(B\Z^l)_{\Z/p}\right)_{G_p} \\& =\rank\Biggl(K_\alpha(B(\Z^n)^{\Z/p})\otimes\bigoplus_{(P)\in\mathcal{P}}\widetilde{R}(P)\Biggr)_{G_p}+\rank \left(K_\alpha(B(\Z^n)^{\Z/p})\otimes K_0(B\Z^l)\right)_{G}.
\end{align*}

In the same way, the rank of the second term, $\left(K_{\alpha+1}(B(\Z^n)^{\Z/p})\otimes K_1(C_r^*(\Z^l\rtimes\Z/p))\right)_{G_p}$, equals

$$\rank\left(K_{\alpha+1}(B(\Z^n)^{\Z/p}\otimes K_1(B\Z^l)_{\Z/p}\right)_{G_p}=\rank \left(K_{\alpha+1}(B(\Z^n)^{\Z/p}\otimes K_1(B\Z^l)\right)_{G}.$$
Then by the Künneth formula for $K$-theory we conclude that
$$\rank K_\alpha(C_r^*(\Z^n\rtimes \Z/p))_{G_p}=\rank \Biggl(K_\alpha(B(\Z^n)^{\Z/p})\otimes\bigoplus_{(P)\in\mathcal{P}}\widetilde{R}(P)\Biggr)_{G_p}+ \rank K_{\alpha}(B\Z^n)_G$$

In the same way as  in the proof of Lemma 6.1 in \cite{SSV24} we have that $$\bigoplus_{(P)\in\mathcal{P}}\widetilde{R}(P)$$ is a permutation $G_p$-module, moreover by Lemma 6.4 in \cite{SSV24} we have an isomorphism of $G_p$-modules

$$\bigoplus_{(P)\in\mathcal{P}}\widetilde{R}(P)\cong \Z[G_p]^{k_l},$$
where $k_l=\frac{p\left(p^{\frac{l}{p-1}}-1\right)}{m}.$

Now, by the Sappiro lemma, we have that
$$\Biggl(K_\alpha(B(\Z^n)^{\Z/p})\otimes\bigoplus_{(P)\in\mathcal{P}}\widetilde{R}(P)\Biggr)_{G_p}\cong \left(K_\alpha(B(\Z^n)^{\Z/p})_{G_p}\right)^{k_l}.$$
Finally, by the Chern character$$\rank\left(K_\alpha(B(\Z^n)^{\Z/p})_{G_p}\right)=\sum_{s-\alpha \text{ even }}\rank\left(H_s(B(\Z^n)^{\Z/p})_{G_p}\right)=\frac{1}{|G_p|}\sum_{s-\alpha\text{ even }}\left(\sum_{t=1}^{|G_p|}e_s(\bar{y}^t_{G_p})\right) $$
$$\rank\left(K_\alpha(B(\Z^n)^{\Z/p})_{G}\right)=\sum_{s-\alpha \text{ even }}\rank\left(H_s(B(\Z^n)^{\Z/p})_{G}\right)=\frac{1}{|G|}\sum_{s-\alpha\text{ even }}\left(\sum_{t=1}^{|G|}e_s(\bar{y}^t_G)\right)$$

Now the result follows.

\end{proof}

\bibliographystyle{alpha} %harvard, unsrt, alpha
\bibliography{myblib}
\end{document}